\documentclass{article}
\usepackage{graphicx} 
\usepackage{amsmath, amsthm, amssymb, calrsfs, wasysym, verbatim, bbm, color, graphics, geometry, enumerate, mathrsfs, dsfont, cancel, stmaryrd}
\usepackage{envmath}
\usepackage{euscript}
\usepackage{cite}
\usepackage{multicol}
\usepackage{appendix}

\allowdisplaybreaks
\def\R{\ensuremath{\mathbb{R}}}

\def\D{\text{div}}
\newtheorem{theorem}{Theorem}

\newtheorem{prop}[theorem]{Proposition}
\newtheorem{Remark}[theorem]{Remark}

\title{Density dependent viscosity for the Poisson-Nernst-Planck-Compressible Navier-Stokes system}
\author{D. Bresch,  M. Kazakova, C. Tonnelier}
\date{January 2026}

\begin{document}
\maketitle

\begin{abstract}
This paper is dedicated to the global existence of entropy weak solutions for the Poisson-Nernst-Planck-Compressible Navier-Stokes system in a periodic domain ${\Pi^d}$ when  the shear viscosity $\mu(\rho) = \mu\, \rho$ with $\mu$ to be constant and $\lambda(\rho) = 0$ assuming a pressure state law singular close to vacuum and a $\gamma$ power elsewhere with $\gamma >1$. It is  important to recall that recently D. Marroquin and D. Wang  (Arxiv 2024) have proved global existence of weak  solutions in the spirit of P.--L. Lions and E. Feireisl when the shear and bulk viscosities are assumed to be constant without  singular pressure part close to vacuum namely $p(\rho) = a \rho^\gamma$ having in hand (from the energy estimate) $u\in L^2(0,T;H^1({\Pi^d}))$. Here the main difficulty is double first because of  the possible degeneracy of the shear viscosity we have  to  derive a formal mathematical entropy generalizing  the BD entropy equality:  Up to the authors knowledge this is completely new.  With this energy equality and the generalization of the BD entropy,  it is then possible to prove global existence of  entropy weak solutions for the system.  An extra difficulty from a physical point of view remains namely if the density vanishes the compressible Navier-Stokes equations  provides no information on the velocity field but the velocity itself is required  in the equation related  to  the concentration. This explain why a singular  pressure state law is needed close to vacuum in the spirit of what has been done in different papers see for instance D. Bresch, B. Desjardins [JMPA 2007], E. Zatorska [J. Diff. Eqs 2012],  D. Bresch, B.~Desjardins and E. Zatorska [JMPA 2015] and references cited therein. 
   \end{abstract} 
   
\medskip

\noindent {\bf Keywords.} Poisson-Nernst-Planck-Navier-Stokes, compressible flows, entropy weak solutions, weak sequential stability, density dependent viscosity, vacuum, singular pressure close to zero.

\section{Introduction and Modeling}
Let us  present the  Poisson-Nernst-Planck-Compressible Navier--Stokes system (PNPCNS) which models the transport of charged particles under the influence of an electrostatic potential in a compressible fluid with shear viscosity $\mu(\rho) = \mu \rho$ and bulk viscosity $\lambda(\rho)=0$ with $\mu>0$ constant. More precisely, we will consider the following system: 
\begin{EqSystem}
    \partial_t \rho + \D(\rho u) = 0, \label{masse}  \\
    \partial_t (\rho u) + \D(\rho u \otimes u) + \nabla p(\rho) -  2 \mu  \D(\rho \text{D}(u)) = \varepsilon \Delta \psi \nabla\psi - \nabla (c^+ + c^-),  \label{momentum}\\
    \partial_t c^+ + \D(c^+ u) = \D(A^+ \nabla c^+) + \D(A^+ e c^+ \nabla \psi), \label{cplus} \\
    \partial_t c^- + \D(c^- u) = \D(A^- \nabla c^-) - \D (A^- e c^- \nabla \psi), \label{cmoins} \\
    - \varepsilon \Delta \psi = e(c^+ - c^-) \label{poisson} 
\end{EqSystem}
with $\mu,e,A^+,A^-,\varepsilon$ strictly positive given constants and
where $\rho$, $u$ and $p = p(\rho)$  denote the density, the velocity field and the pressure of the fluid. We denote $\text{D}(u)$ the usual symmetric tensor $\text{D}(u) =({\nabla u + {}^T \nabla u})/{2}$ and we focus on the pressure state law $p'(\rho) = c_1 \rho^{\gamma-1}$ for $\rho>1$ with $\gamma > 1$ and $p'(\rho) = c_2 \rho^{-4k-1}$ for $\rho \le 1$ with $k>1$ and $c_1$, $c_2>0$ two given constants. Moreover the positive quantities  $c^+$ and  $c^-$ are respectively the concentrations of cations and anions and $\psi$ is the electrostatic potential.  The system is also endowed with initial boundary conditions namely
\begin{equation}\label{initial} 
\rho\vert_{t=0} = \rho_0 \ge 0, \qquad (\rho u)\vert_{t=0} = m_0, 
   \qquad c^+\vert_{t=0} = c^+_0\ge 0,    \qquad  c^-\vert_{t=0} = c^-_0\ge 0.
\end{equation} 
We will consider a periodic domain $\Omega = \Pi^d$ with $d=3$ (Note that the result holds for $d=2$). As mentioned in \cite{BrDeZa}, from the mathematical point of view, a  singular pressure close to vacuum  has some stabilizing properties. Such singular pressure has been introduced in \cite{BrDe2} and has been  used in \cite{MuPoZa} to investigate the model of compressible mixture. Steady compressible Navier--Stokes system with pressure singular at vacuum was also investigated by \cite{La}. For studies on doubly singular pressure in the context of mixtures we refer to paper by \cite{FeLuMa}. See also the work \cite{KiLaNi}  where they study lubrication equations in the presence of strong slippage.

The Poisson-Nernst–Planck system is a simplified continuum model for a wide variety of chemical, physical and biological applications. It is used to describe the ionic transport and electrostatic interactions. This model couples the Poisson equation for the electrostatic potential with the Nernst-Planck equations that describe the fluxes of ionic species under the influence of both
electrical and concentration gradients. Such system may be coupled 
to Navier-Stokes or Stokes systems if we are interested by the dynamic of a viscous conductive fluid with the crowded charged particles.

As described in \cite{MaWa}, there is a lot of literature on the incompressible version of the PNPNS equations (see \cite{CoIg} and references cited therein) but regarding the compressible case of the system, the literature is more limited, mostly focusing on smooth solutions in the whole space, which are either local or small; see the works in
\cite{ShWa}, \cite{WaLiTa}, \cite{WaLiTa1} and \cite{ToTa}. In this paper we investigate the global weak solutions to the initial-boundary value problem of the compressible PNPNS system with density dependent viscosities in a periodic domain. This is some sense extend to this case the nice paper by \cite{MaWa} which was dedicated to the constant viscosities case. Note that this case is completely different from the mathematical point of view: The constant viscosities case  corresponds to the framework defined in \cite{Li},
\cite{FeNoPe}, \cite{Fe}. See also \cite{BrJa} if interested by non-monotone pressure laws or constant anisotropic viscous coefficients. The density dependent case corresponds to the framework initiated in \cite{BrDe1}, \cite{BrDe2} with a full mathematical understanding with \cite{VaYu}, \cite{BrVaYu}, \cite{LiXi}. 

\medskip

In this paper, we first start with Section 2 to precise the definition of the entropy weak solutions for the system and the main result. Then in Section 3, we recall how to get the formal energy equality and then we provide a non-trivial extension of the BD entropy equality for the Poisson--Nernst--Planck compressible Navier--Stokes equations with the density dependent viscosities
$\mu(\rho) = \mu\, \rho$ and $\lambda(\rho)=0$. In Section 4, we present the approximate system for which we may construct regularized solutions. Note that we explain the method of construction without proof since it follows a standard method. In Section 5, we explain how to let the regularized parameters tend to zero to prove the main result announced in Section 2. Having in hand the Energy estimates coupled with the BD entropy generalization is the cornerstone.

\section{Definition and main result} We consider problem \eqref{masse}--\eqref{poisson} posed on a periodic domain  and assume that the initial data satisfy the following conditions
\begin{align}
& \rho_0 \, e(\rho_0) \in L^1(\Pi^d), \qquad \nabla \sqrt\rho_0 \in L^2(\Pi^d) \qquad  \hbox{ with } \rho_0 \ge 0
\label{init1}  \\
& m_0 \in L^1(\Pi^d) \hbox{ with } m_0 = 0 \hbox{ if } \rho_0 = 0 \hbox{ and } 
    |m_0|^2/\rho_0 \in L^1(\Pi^d ) 
    \label{init2} \\
&  c^\pm_0 \ge 0, \qquad c^\pm_0 \in L^1(\Pi^d), \qquad \sigma(c^\pm_0) \in L^1(\Pi^d), \qquad \rho_0\sigma(c^\pm_0/\rho_0) 
  \in L^1(\Pi^d).
\label{init3} 
\end{align} 
where $\sigma(s)=s\ln s - s +1$ and  $\rho^2 de/d\rho=p(\rho)$ where we choose the following pressure state law defined by 
\begin{align}{\label{pressure}}
p(\rho)
& = - c_1 \rho^{-4k} \hbox{ for } \rho \le 1 \nonumber \\
& = c_2 \rho^{\gamma} \hbox{ for } \rho > 1 
\end{align}
with $c_1,c_2>0$ given with $\gamma>1$ and where the constant $k\ge 1$ (note that $e$ is piecewise defined with a constant reference state $1$). We say that $(\rho,u, c^+,c^-,\psi)$ is a  $\kappa$ entropy weak solution of System 
\eqref{masse}--\eqref{poisson} with \eqref{pressure}  with the initial conditions \eqref{initial} (assuming \eqref{init1}--\eqref{init3}) for $t\in [0,T]$ with $T>0$ a given constant and $x\in \Pi^d$ (periodic boundary condition) if
\begin{itemize}
\item The density $\rho$ is nonnegative and 
$$\rho \in {\mathcal C}([0,T]; L^\gamma (\Pi^d)) \hbox{ weak }, \qquad
    \sqrt\rho \in L^\infty((0,T); H^1( \Pi^d))\cap L^\infty(0,T;L^\gamma(\Pi^d)), \qquad 
$$
with $\rho\vert_{t=0} = \rho_0$.
\item The velocity field $u$ satisfies
$$\sqrt \rho \nabla u \in L^2((0,T)\times \Pi^d), \qquad
 \rho u \otimes u \in L^\infty(0,T; \L^1(\Pi^d)), \qquad 
 \rho u \in {\mathcal C}([0,T]; L^{2\gamma/(\gamma+1)}(\Pi^d) \hbox{ weak})
$$
with $(\rho u)\vert_{t=0} = m_0$ and the following bound
$$u \in L^p(0,T;W^{1,q}(\Pi^d))$$
with $1/p=1/2+1/8k$ and $1/q=1/2+1/24k$.
\item The ion densities $c^+$ and $c^-$ are nonnegative  and
$$c^\pm \in L^\infty(0,T;L^1(\Pi^d))\cap L^1(0,T;W^{1,3/2}(\Pi^d)), \qquad
  \nabla  \sqrt c^\pm \in L^2((0,T)\times\Pi^d).
$$
\item The electrostatic potential $\psi$ satisfies 
$$\psi \in L^\infty(0,T;H^1(\Pi^d))\cap L^1((0,T);W^{3,3/2}(\Pi^d))\cap {\mathcal C}([0,T;] L^p(\Pi^d) \hbox{ weak }) \hbox{ for }   p \in [1,6)
$$
\item The continuity and momentum equations are satisfies in the distributions sense.
\item The equations on $c^+$ and $c^-$ and the initial condition are satisfied  in the distribution sense .
\item The function $\psi$ is a strong solution of the electrostatic equation .
\item The solution satisfies the following energy inequality for all $t\in [0,T]$
\begin{align}{\label{energy1}}
 E_1(t) &+ \int_0^t \int_{\Pi^d}  A^+ c^+ |\nabla(\ln c^+ + e \psi)|^2 dx + \int_0^t \int_{\Pi^d} A^- c^- |\nabla(\ln c^- - e \psi)|^2  dx  \nonumber \\
 & +2 \int_0^t \int_{\Pi^d}  \mu\, \rho |D(u)|^2 \, dx
\leq E_1(0) 
\end{align}  
with  
\begin{equation}{\label{energy2}}
E_1(t) = \int_{\Pi^d} \left( \rho \left( \frac{1}{2} |u|^2 + e(\rho) \right) + \sigma (c^+) + \sigma (c^-) + \frac{\varepsilon}{2} |\nabla \psi|^2 \right) dx. 
\end{equation} 
where $\sigma (c)= c \ln c -c + 1$ and $\rho^2 de(\rho)/d\rho =p(\rho).$
\item The solutions satisfies also the following BD entropy inequality for all $t\in [0,T]$
\begin{align}{\label{BD entropy final}}
  E_2(t) 
    &+ 8 \mu \int_0^t \int_{\Pi^d} |\nabla \sqrt{c^+}|^2 + 8 \mu \int_0^t \int_{\Pi^d} |\nabla \sqrt{c^-}|^2 + 2 \mu \int_0^t \int_{\Pi^d} \varepsilon |\Delta \psi|^2  \nonumber \\
    & + 2 \int_0^t\int_{\Pi^d} \mu\, \rho |A(u)|^2  
     + 2 \mu \int_0^t\int_{\Pi^d} \frac{p'(\rho) }{\rho} |\nabla \rho|^2
     \\
     &+ \int_0^t\int_{\Pi^d} A^+ c^+ |\nabla (\ln c^+ + e \psi)|^2 + \int_0^t\int_{\Pi^d} A^- c^- |\nabla (\ln c^- - e \psi|^2 
    \le   E_2(0) \nonumber
\end{align}
where
\begin{align}{\label{Entrop} }
&  E_2 (t) =    \frac{1}{2}   \int_{\Pi^d} \rho |u + 2\mu \nabla \ln\rho|^2
     +  \int_{\Pi^d} \rho e(\rho) +  \int_{\Pi^d} \varepsilon \frac{| \nabla \psi|^2}{2} \\
    & 
    \hskip1cm +  \int_{\Pi^d} \left( \sigma (c^+)
   + \sigma (c^-) \right) \nonumber 
 + 2\mu\, \int_{\Pi^d} \left( \frac{1}{A^+}\rho
    \sigma(c^+/\rho) 
     + \frac{1}{A^-} \rho \sigma(c^-/\rho) \right) 
\end{align}
\end{itemize}
with $\sigma (c)= c \ln c -c + 1$ and $\rho^2 de(\rho)/ d \rho =p(\rho)$.

\begin{Remark} This is important to observe that for our density dependent case, we have a new mathematical entropy generalizing perfectly the BD entropy. In the constant viscosities case, an estimate exists {\rm (see \cite{MaWa}: Estimate (1.16) in \rm Theorem 1)} but not such a kind of precise formal equality.  This is a nice property for instance to propose a relative entropy based on the energy and the BD entropy. This is really helpful to precise singular limit such as the electro-neutral asymptotic: See for instance the forthcoming paper {\rm \cite{To}}. This may be also used to design appropriate numerical schemes.
\end{Remark}

\bigskip

\begin{theorem} {\rm (Existence of global entropy weak solutions)}. Let $\gamma>1$ and let the  initial data satisfy 
\eqref{init1}--\eqref{init3}.
Then, for any $T>0$, there exists a finite entropy weak solution of  \eqref{masse}--\eqref{poisson} with initial conditions \eqref{initial} and the pressure state law \eqref{pressure} in the sense of the definition given just above.
\end{theorem}

We prove the existence of global entropy weak solutions as a limit of solutions of a regularized system of equations by combining the theory for the Navier-Stokes equations with density dependent viscosities  with  results regarding the Poisson-Nernst-Planck (PNP)  subsystem  following in some sense some arguments that may be found in \cite{MaWa} for this part. The {\it a priori}estimates coming from the energy coupled strongly with the {\it apriori} estimates coming from the generalization of the BD entropy  allow to perform the  nonlinear weak stability process for the Poisson-Nernst-Planck density dependent compressible Navier-Stokes system.

\section{Some important equalities at formal level} 

\subsection{Energy equality at formal level}
Let us provide a formal energy  equality for the system.  This equality may be proved in similar manner than the one for the Poisson-Nernst-Planck-Compressible Navier-Stokes system for constant viscosities as derived for instance by D. Marroquin and D. Wang: See  \cite{MaWa}.  We will recall the derivation for reader's convenience because it is not straightforward at the first time.
\begin{prop}
    Let $(u,\rho,c^\pm,\psi)$ a smooth solution of \eqref{masse}--\eqref{poisson} with initial conditions \eqref{initial} then it satisfies the following energy  equality 
    \begin{equation}\label{eq énergie}
    \frac{d}{dt} E_1(t) + \int_{\Pi^d} \left( 2 \mu\, \rho  |D(u)|^2  + A^+ c^+ |\nabla (\ln c^+ + e \psi)|^2 + A^-c^- |\nabla (\ln c^- - e \psi|^2 \right) = 0
    \end{equation}
    with $$E_1(t) = \int_{\Pi^d} \left( \rho \left( \frac{1}{2} |u|^2 + e(\rho) \right) + \sigma (c^+) + \sigma(c^-) + \frac{\varepsilon}{2}|\nabla \psi|^2  \right)$$
    and $\sigma(c) = c \ln c - c + 1$.
\end{prop}

\begin{proof}  \noindent {\bf The compressible Navier-Stokes contribution:} 
Let us take the scalar product of the momentum equation with $u$ and use the mass equation to write 
\begin{align*}   
u \cdot \( \partial_t (\rho  u) + \D (\rho u \otimes u) \right)
    & = u \cdot \partial_t (\rho u) + \left( (\rho u \cdot \nabla ) u \right) \cdot u + \D (\rho u) \cdot u^2 \\ 
    &= u \cdot \left( \rho \partial_t u + u \partial_t \rho \right) + (\rho u \cdot \nabla) u^2 + \D (\rho u) \cdot u^2 \\
    & = \rho \partial_t \frac{|u|^2}{2} + \left( (\rho u \cdot \nabla) u \right) \cdot u \\
        & = \partial_t \left( \rho \frac{|u|^2}{2} \right) - \frac{1}{2} |u|^2 \partial_t \rho + \D (\rho u \frac{1}{2} |u|^2) 
    - \frac{1}{2}|u|^2 \D(\rho u) \\
    &= \partial_t \left( \rho \frac{|u|^2}{2}  \right) + \D(u \rho \frac{1}{2} |u|^2).
\end{align*}

Let us recall that formally
\begin{align*}
    \partial_t \pi(\rho) + \D(u \pi(\rho)) + p(\rho) \D u = 0.
\end{align*}
choosing $\pi$ such that $\pi'(\rho) \rho - \pi(\rho) = p(\rho)$. This is equivalent to write $$\pi(\rho) = \rho \left( \int^\rho \frac{p(s)}{s^2} ds \right) := \rho e(\rho).$$
Moreover  $$u \cdot \nabla p = \D (u p(\rho)) - p \D(u).$$
In that way, summing we get
\begin{align*}
   & \partial_t \pi(\rho) + \D(u \pi(\rho)) + \D(u p(\rho)) - u\cdot \nabla p(\rho) \\
    &= \partial_t (\rho e(\rho)) + \D(u \rho e(\rho)) + \D(u p(\rho)) - u\cdot \nabla p(\rho).
\end{align*}
therefore $$u \cdot \left[ \partial_t (\rho u) + \D(\rho u \otimes u) + \nabla p(\rho) \right] = \partial_t \left( \rho ( \frac{1}{2} |u|^2 + e(\rho)) \right) + \D \left( u (\rho (\frac{1}{2} |u|^2 + e(\rho)) + p(\rho)) \right)$$ with  $e(\rho) = \int^\rho \frac{p(s)}{s^2} ds$. 
Then it comes 
\begin{align*}
    \partial_t \left( \rho ( \frac{1}{2} |u|^2 + e(\rho)) \right) & + \D \left( u (\rho (\frac{1}{2} |u|^2 + e(\rho)) + p(\rho)) \right) \\ 
    &= u \cdot \left[ \D (2 \mu \, \rho  D(u))  \right] + u \cdot \left[\varepsilon \Delta \psi \nabla \psi \right] - u \cdot \left[\nabla c^+ + \nabla c^- \right].
\end{align*} 
Therefore integrating in space and by parts, we get 
\begin{align*}
    \partial_t \int_{\Pi^d}  \left( \rho ( \frac{1}{2} |u|^2 + e(\rho)) \right) & +  2 \mu \int_{\Pi^d} \rho |D(u)|^2 \\ 
    &= \int_{\Pi^d} u \cdot \left[\varepsilon \Delta \psi \nabla \psi \right] - \int_{\Pi^d}  u \cdot \left[\nabla c^+ + \nabla c^- \right].
\end{align*}

\bigskip
\noindent {\bf The Poisson-Nernst-Planck contribution:} 
We define the functions  $\sigma_j$ such that $s \sigma_j '(s) - \sigma_j (s) = \varphi_j (s)$ (where $\varphi^+ (c^+) = c^+$ and $\varphi^- (c^-) = c^-$). 
Let us multiply the equation on $c^+$ by $\sigma^{'}_+ (c^+) + e \psi$, we get : 
$$\partial_t c^+ (\sigma^{'}_+ (c^+) + e \psi) + \D(c^+ u)(\sigma^{'}_+ (c^+) + e \psi) = \D(A^+ \nabla c^+ + A^+ e c^+ \nabla \psi) (\sigma^{'}_+ (c^+) + e \psi)$$
which gives, 
$$\partial_t \sigma_+ (c^+) + e\partial_t c^+ \psi = - \D(c^+ u)(\sigma^{'}_+ (c^+) + e \psi) + \D(A^+ \nabla \varphi^+ (c^+) + A^+ e c^+ \nabla \psi)(\sigma^{'}_+ (c^+) + e \psi) := \alpha + \beta$$
Let us remark that $\nabla \varphi^+(c^+) = \nabla \left( \sigma^{'}_+ (c^+) c^+ - \sigma_+ (c^+) \right) = c^+ \nabla \sigma^{'}_+ (c^+)$, thus
\begin{align*}
    \beta &= \D(A^+ \nabla \varphi^+ (c^+) + A^+ e c^+ \nabla \psi)(\sigma^{'}_+ (c^+) + e \psi) \\ 
    &= \D \left(A^+ c^+ \nabla(\sigma^{'}_+ (c^+) + e \psi)(\sigma^{'}_+ (c^+) + e \psi)\right) - A^+ c^+ |\nabla(\sigma^{'}_+ (c^+) + e \psi)|^2 
\end{align*}
and 
\begin{align*}
    \alpha &= - \D(c^+ u)(\sigma^{'}_+ (c^+) + e \psi) \\
    &= - \D \left( c^+ (\sigma^{'}_+ (c^+) + e \psi)u \right) + c^+ \nabla(\sigma^{'}_+ (c^+) + e \psi) \cdot u \\ 
    &= - \D \left(c^+ (\sigma^{'}_+ (c^+) + e \psi)u \right) + \nabla \varphi^+ (c^+) \cdot u + ec^+ \nabla \psi \cdot u.
\end{align*}
We then get 
\begin{align}{\label{sigma c+}}
    \partial_t \sigma_+(c^+) + e \partial_t c^+ \psi & = - \D \left( c^+ (\sigma^{'}_+ (c^+) + e \psi)u \right) + \nabla \varphi^+ (c^+) \cdot u + ec^+\nabla \psi \cdot u \\
    & + \D \left(A^+ c^+ \nabla(\sigma^{'}_+ (c^+) + e \psi)(\sigma^{'}_+ (c^+) + e \psi)\right) - A^+ c^+ |\nabla(\sigma^{'}_+ (c^+) + e \psi)|^2. 
\end{align}
In the same manner, we multiply the equation on   $c^-$ by $\sigma^{'}_- (c^-) - e \psi$, we get: 
\begin{align}
    \partial_t \sigma_-(c^-) - e \partial_t c^- \psi &= - \D \left( c^- (\sigma^{'}_- (c^-) - e \psi)u \right) + \nabla \varphi^- (c^+) \cdot u - ec^-\nabla \psi \cdot u {\label{sigma c-}} \\
    &+ \D \left(A^- c^- \nabla(\sigma^{'}_- (c^-) - e \psi)(\sigma^{'}_- (c^-) - e \psi)\right) - A^- c^- |\nabla(\sigma^{'}_- (c^-) - e \psi)|^2. 
\end{align}
Summing the equations   (\ref{sigma c+}) and (\ref{sigma c-}), we get
\begin{align*}
    \partial_t &\left( \sigma_+(c^+)  + \sigma_-(c^-) \right)  + e \partial_t (c^+- c^-) \psi + A^+c^+ |\nabla(\sigma^{'}_+ (c^+) + e \psi)|^2 + A^- c^- |\nabla(\sigma^{'}_- (c^-) - e \psi)|^2 \\ 
    &= - \D \left( c^+ (\sigma^{'}_+ (c^+) + e \psi)u \right) - \D \left( c^- (\sigma^{'}_- (c^-) - e \psi)u \right) + \left( \nabla \varphi^+ (c^+) + \nabla \varphi^- (c^-) \right) \cdot u + e(c^+ - c^-) \nabla \psi \cdot u \\ 
    &+ \D \left(A^+ c^+ \nabla(\sigma^{'}_+ (c^+) + e \psi)(\sigma^{'}_+ (c^+) + e \psi)\right) + \D \left(A^- c^- \nabla(\sigma^{'}_- (c^-) - e \psi)(\sigma^{'}_- (c^-) - e \psi)\right).
\end{align*}
Note that  $e \partial_t (c^+ - c^-) = - \varepsilon \Delta \partial_t \psi$ and $- \Delta \partial_t \psi \psi = \frac{1}{2} \partial_t (|\nabla \psi|^2) - \D(\psi \nabla \partial_t \psi) $, thus we get
\begin{align*}
    \partial_t &\left( \sigma_+ (c^+) + \sigma_- (c^-) \right) + \frac{\varepsilon}{2} \partial_t (|\nabla \psi|^2) + A^+c^+ |\nabla(\sigma^{'}_+ (c^+) + e \psi)|^2 + A^- c^- |\nabla(\sigma^{'}_- (c^-) - e \psi)|^2 \\ 
    &= \varepsilon \D(\psi \nabla \partial_t \psi) - \D \left( c^+ (\sigma^{'}_+ (c^+) + e \psi)u \right) - \D \left( c^- (\sigma^{'}_- (c^-) - e \psi)u \right) + \left( \nabla \varphi^+ (c^+) + \nabla \varphi^- (c^-) \right) \cdot u \\ 
    & - \varepsilon \Delta \psi \nabla \psi \cdot u + \D \left(A^+ c^+ \nabla(\sigma^{'}_+ (c^+) + e \psi)(\sigma^{'}_+ (c^+) + e \psi)\right) + \D \left(A^- c^- \nabla(\sigma^{'}_- (c^-) - e \psi)(\sigma^{'}_- (c^-) - e \psi)\right). 
\end{align*}
We recall that 
\begin{align*}
    \partial_t \left( \rho ( \frac{1}{2} |u|^2 + e(\rho)) \right) & + \D \left( u (\rho (\frac{1}{2} |u|^2 + e(\rho)) + p(\rho)) \right) \\ 
    &= u \cdot \D (2 \mu  \, \rho  D(u))  + u \cdot \left[\varepsilon \Delta \psi \nabla \psi \right] - u \cdot \left[\nabla c^+ + \nabla c^- \right].
\end{align*}
Thus, summing the last two equalities, we get :
\begin{align*}
    \partial_t  & \left( \rho \left( \frac{1}{2} |u|^2 + e(\rho) \right) + \sigma_+ (c^+) + \sigma_- (c^-) + \frac{\varepsilon}{2} |\nabla \psi|^2 \right) + A^+c^+ |\nabla(\sigma^{'}_+ (c^+) + e \psi)|^2 + A^- c^- |\nabla(\sigma^{'}_- (c^-) - e \psi)|^2 \\ 
    &= - \D \left( u (\rho (\frac{1}{2} |u|^2 + e(\rho)) + p(\rho)) \right) + u \cdot \D (2 \mu\, \rho D(u))  + \varepsilon \D(\psi \nabla \partial_t \psi) \\
    & - \D \left( c^+ (\sigma^{'}_+ (c^+) + e \psi)u \right)
    - \D \left( c^- (\sigma^{'}_- (c^-) - e \psi)u \right) + \D \left( A^+ (\nabla \varphi^+(c^+) + c^+ e \psi) (\sigma^{'}_+ (c^+) + e \psi) \right) \\
    &+ \D \left( A^- (\nabla \varphi^-(c^-) - c^- e \psi) (\sigma^{'}_- (c^-) - e \psi) \right).
\end{align*}
Integrating on ${\Pi^d}$, considering periodic boundary conditions, we get 
$$\frac{d}{dt} E(t) + \int_{\Pi^d} \left( A^+ c^+ |\nabla(\ln c^+ + e \psi)|^2 + A^- c^- |\nabla(\ln c^- - e \psi)|^2 \right) dx = \int_{\Pi^d}  u \cdot  \D (2 \mu \, \rho D(u))  dx $$ 
with  $$E(t) = \int_{\Pi^d} \left( \rho \left( \frac{1}{2} |u|^2 + e(\rho) \right) + \sigma_+ (c^+) + \sigma_- (c^-) + \frac{\varepsilon}{2} |\nabla \psi|^2 \right) dx. $$

\noindent Finally, integrating by parts, we have
$$\int_{\Pi^d}  u \cdot \D (2 \mu\, \rho D(u))  dx =  - 2 \int_{\Pi^d} \mu\, \rho |D(u)|^2 dx$$
and therefore the energy equality (\ref{eq énergie}).
\end{proof}

\begin{Remark}
 We could at the level of Energy estimates consider the particular case  $\D (\rho \text{D}(u))$ or the general case  $2 \D (\mu(\rho) \text{D}(u)) + \nabla (\lambda (\rho) \D(u))$. The particular case  $\D(\rho \text{D}(u))$ summarizes to choose $\mu(\rho) = \rho$ and $\lambda(\rho) = 0$.
 \end{Remark}

\subsection{Extension of the BD entropy at formal level}
We will do the calculation for $\mu(\rho) = \mu \, \rho$ and $\lambda(\rho) = 0$. We want to get an extra estimate allowing to bound derivative in space for the density $\rho$ as required for the compressible Navier-Stokes equations with density dependent viscosities and to bound derivative in space for the concentration $c^\pm$ as in papers by D. Maroquin  and D. Wang \cite{MaWa}. The main novelty here is that we really get a new mathematical entropy for Poisson-Nernst-Planck Compressible Navier-Stokes equations with a degenerate shear viscosity  that seems to be not known. Note that this BD entropy type equality does not applied in the constant viscosities case where only estimates may be obtained. 

\begin{Remark}
For the density dependent viscosity case, we really need an extension of the BD entropy which provides appropriate bounds on   $\rho$ and derivatives. Therefore, we follow the calculations made for the degenerate compressible Navier-Stokes equation  and see what would happen with the new quantities linked to $c^\pm$, $\psi$.  Magically, we will see that everything goes well and that we will get a precious supplementary estimate which will provide new estimates on the density, the velocity field and the concentrations. 
\end{Remark} 

More precisely let us prove the following proposition
\begin{prop} Let $(u,\rho,c^\pm,\psi)$ a smooth solution of \eqref{masse}--\eqref{poisson} with initial conditions \eqref{initial} then
\begin{align}{\label{BD entropy final}}
 \frac{d}{dt} E_2(t) + 
    &+ 8 \mu \int_{\Pi^d} |\nabla \sqrt{c^+}|^2 + 8 \mu \int_{\Pi^d} |\nabla \sqrt{c^-}|^2 + 2 \mu \int_{\Pi^d} \varepsilon |\Delta \psi|^2  + 2 \int_{\Pi^d} \mu\, \rho\,  |A(u)|^2 + 2\mu  \int_{\Pi^d} \frac{p'(\rho) }{\rho} |\nabla \rho|^2  \nonumber \\
    &+ \int_{\Pi^d} A^+ c^+ |\nabla (\ln c^+ + e \psi)|^2 + \int_{\Pi^d} A^- c^- |\nabla (\ln c^- - e \psi|^2 = 0 
\end{align}
where
\begin{align}{\label{BD entropy final2}}
 E_2(t) & =  \frac{1}{2} \int_{\Pi^d} \rho |u + 2\mu \nabla \log\rho|^2  +  \int_{\Pi^d} \rho e(\rho) +  \int_{\Pi^d} \varepsilon \frac{| \nabla \psi|^2}{2} \nonumber \\
    & +  \int_{\Pi^d} \left( \sigma (c^+) + \sigma (c^-) \right) 
     + 2\, \mu  \int_{\Pi^d} \left( \frac{1}{A^+}\rho \sigma( \frac{c^+}{\rho}) + \frac{1}{A^-} \rho \sigma(\frac{c^-}{\rho})  \right) \nonumber \\
\end{align}
\end{prop}

\bigskip

\begin{proof} {\bf The compressible Navier-Stokes contribution:} 
Taking the gradient of the  mass equation, we get
$$\partial_t(\nabla\rho) + {\rm div}(\nabla \rho\otimes u) + {\rm div}(\rho {}^t\nabla u) = 0$$
that we can write 
$$\partial_t(\rho v) + {\rm div}(v \otimes u) + {\rm div}(\mu \rho {}^t\nabla u) = 0$$
where 
$v= \mu \nabla\log\rho$.
Multiplying by $2$ and adding to the momentum equation on $u$, we get 
\begin{equation}\label{eq u+v choix f}
    \partial_t (\rho (u+2v)) + \D( \rho u \otimes (u+2v)) - 2 \D(\mu \rho A(u)) + \nabla p (\rho) = - c^+ \nabla (\ln c^+ + e \psi) - c^- \nabla (\ln c^- - e \psi)
\end{equation}
where $A(u) = (\nabla u - {}^t\nabla u)/2$. Then, taking the scalar product of (\ref{eq u+v choix f}) with $(u+2v)$ and using the mass equation tested by $|u+2v|^2/{2}$, we get 
\begin{align*}
    (u+2v) \cdot \partial_t (\rho (u+2v)) & + (u+2v)\cdot  \D( \rho u \otimes (u+2v)) - 2 (u+2v)\cdot  \D(\mu \rho A(u)) + (u+2v) \cdot \nabla p(\rho) \\
    &- \frac{|u+2v|^2}{2} \partial_t \rho - \frac{|u+2v|^2}{2} \D(\rho u) \\
    &= - c^+ \nabla (\ln c^+ + e \psi) \cdot (u+2v) - c^- \nabla (\ln c^- - e \psi) \cdot (u+2v)
\end{align*}
so, 
\begin{align*}
    \partial_t \left( \rho \frac{|u+2v|^2}{2} \right) & + \D \left( \rho u \otimes \frac{|u+2v|^2}{2} \right) - 2 (u+2v) \D(\mu \rho A(u)) + (u+2v) \nabla p(\rho) \\
    &= - c^+ \nabla (\ln c^+ + e \psi) \cdot (u+2v) - c^- \nabla (\ln c^- - e \psi) \cdot (u+2v)
\end{align*}

\noindent Integrating this equality over $\Omega$, we have 
\begin{align}
    \frac{d}{dt} \int_{\Pi^d} & \rho \frac{|u+2v|^2}{2} - 2 \mu \int_{\Pi^d} (u+2v)\cdot  \D(\rho A(u)) + \int_{\Pi^d} (u+2v)\cdot  \nabla p(\rho) \nonumber \\
    &= - \int_{\Pi^d} c^+ \nabla (\ln c^+ + e \psi) \cdot (u+2v) - \int_{\Pi^d} c^- \nabla (\ln c^- - e \psi) \cdot (u+2v) 
\end{align}

Let consider the pressure term, 
\begin{align*}
    (u+2v) \cdot \nabla p(\rho) &= u \cdot \nabla p(\rho) + 2\mu \nabla \ln\rho \cdot \nabla p(\rho) \\
    &= u \cdot \nabla p(\rho) + 2\mu \frac{1}{\rho} \nabla \rho \cdot \nabla p(\rho) \\
    &= u \cdot \nabla p(\rho) + 2 \mu\frac{ p'(\rho)}{\rho} |\nabla \rho|^2
\end{align*}
Thus, 
\begin{align*}
    \int_{\Pi^d} (u+2v) \cdot \nabla p(\rho) & = \int_{\Pi^d} u \cdot \nabla p(\rho) + 2 \mu \int_{\Pi^d} \frac{ p'(\rho)}{\rho} |\nabla \rho|^2 \\
    &= \int_{\Pi^d} u \cdot \nabla (\pi'(\rho)) \rho + 2 \mu \int_{\Pi^d} \frac{ p'(\rho)}{\rho} |\nabla \rho|^2 \\
    &= - \int_{\Pi^d} \pi'(\rho) \D(\rho u) + 2 \mu \int_{\Pi^d} \frac{ p'(\rho)}{\rho} |\nabla \rho|^2 \\
    &= \int_{\Pi^d} \pi'(\rho) \partial_t \rho + 2 \mu \int_{\Pi^d} \frac{ p'(\rho)}{\rho} |\nabla \rho|^2 \\
    & = \frac{d}{dt} \int_{\Pi^d} \pi(\rho) + 2 \mu \int_{\Pi^d} \frac{p'(\rho)}{\rho} |\nabla \rho|^2
\end{align*}
where $\Pi(\rho) = \rho \int_{\bar{\rho}}^\rho \frac{p(s)}{s^2} ds = \rho e(\rho)$ with $\bar{\rho}$ a constant reference density. For the viscous term,
$$\int_{\Pi^d} v \cdot \D(\mu \rho A(u)) = 2 \int_{\Pi^d} \nabla \ln\rho \D(\mu \rho A(u)) = 0$$ as $\nabla \ln \rho$ is a gradient. 
Moreover 
\begin{align*}
    \int_{\Pi^d} u\cdot  \D(\mu \rho A(u)) &= - \int_{\Pi^d} \nabla u : (\mu \rho A(u)) \\
    &= - \mu \int_{\Pi^d} \rho |A(u)|^2
\end{align*}

\noindent Finally we obtain the equality 
\begin{align}\label{BD entropy u + v}
    \frac{1}{2} \frac{d}{dt} & \int_{\Pi^d} \left( \rho |u + 2\mu \nabla \ln \rho|^2 + 2 \rho e(\rho) \right) + 2 \mu \int_{\Pi^d} \rho |A(u)|^2 + 2 \mu \int_{\Pi^d} \frac{ p'(\rho)}{\rho} |\nabla \rho|^2 \nonumber \\
    &= - \int_{\Pi^d} c^+ \nabla (\ln c^+ + e \psi) \cdot (u + 2\mu \nabla \ln\rho) - \int_{\Pi^d} c^- \nabla (\ln c^- - e \psi) \cdot (u + 2\mu \nabla\ln\rho )
\end{align}
Let's see how to write the terms of the right hand side
\begin{align*}
    - \int_{\Pi^d} c^+ u\cdot  \nabla (\ln c^+ + e \psi) &= \int_{\Pi^d} \D(c^+ u) (\ln c^+ + e \psi) \\
    &= \int_{\Pi^d} (\ln c^+ + e \psi) \left[ \D(A^+ c^+ \nabla (\ln c^+ + e \psi) - \partial_t c^+ \right] \\
    &= - \int_{\Pi^d} A^+ c^+ |\nabla (\ln c^+ + e \psi)|^2 - \int_{\Pi^d} \partial_t c^+ (\ln c^+ + e \psi) \\
    &= - \int_{\Pi^d} A^+ c^+ |\nabla (\ln c^+ + e \psi)|^2 - \int_{\Pi^d} \left( \partial_t (c^+ \ln c^+ - c^+ +1) - e\psi \partial_t c^+ \right).
\end{align*}

\noindent Doing the same for $- \int_{\Pi^d} c^- u \nabla (\ln c^- - e \psi)$, and combining, we have 
\begin{align*}
    - \int_{\Pi^d} c^+ u\cdot  \nabla (\ln c^+ + e \psi) & - \int_{\Pi^d} c^- u \cdot \nabla (\ln c^- - e \psi) = - \int_{\Pi^d} A^+ c^+ |\nabla (\ln c^+ + e \psi)|^2 - \int_{\Pi^d} A^- c^- |\nabla (\ln c^- - e \psi)|^2 \\
    & - \frac{d}{dt} \int_{\Pi^d} 
      \bigl[(c^+ \ln c^+ - c^+ +1) + (c^- \ln c^- - c^- +1)\bigr] 
      - \int_{\Pi^d} e\psi \partial_t (c^+ - c^-).
\end{align*}

\noindent Using the Poisson equation, 
\begin{align*}
    - \int_{\Pi^d} e\psi \partial_t (c^+ - c^-) &= \int_{\Pi^d} \varepsilon \psi \partial_t (\Delta \psi) \\
    &= - \int_{\Pi^d} \varepsilon \partial_t \left( \frac{|\nabla \psi|^2}{2} \right)
\end{align*} 

\noindent Then, (\ref{BD entropy u + v}) writes as 
\begin{align}{\label{BD entropy f(rho)}}
    \frac{1}{2} \frac{d}{dt}  \int_{\Pi^d} \rho |u + 2\mu\nabla \ln\rho|^2 & + \frac{d}{dt} \int_{\Pi^d} \rho e(\rho) + 2 \mu \int_{\Pi^d} \rho |A(u)|^2 + 2 \mu \int_{\Pi^d} \frac{p'(\rho)}{\rho} |\nabla \rho|^2 + \frac{d}{dt} \int_{\Pi^d} \varepsilon \frac{| \nabla \psi|^2}{2} \nonumber \\
    &+ \frac{d}{dt} \int_{\Pi^d} \left( \sigma (c^+) + \sigma (c^-) \right)  + \int_{\Pi^d} A^+ c^+ |\nabla (\ln c^+ + e \psi)|^2 + \int_{\Pi^d} A^- c^- |\nabla (\ln c^- - e \psi)|^2 \nonumber \\
    &= - 2 \int_{\Pi^d} c^+ \nabla (\ln c^+ + e \psi) \nabla \ln \rho - 2 \int_{\Pi^d} c^- \nabla (\ln c^- - e \psi) \nabla \ln \rho.
\end{align}

\noindent We now have to express the two remaining terms differently to obtain the desired inequality wanted. 

{\bf The Poisson-Nernst-Planck contribution:} The Nernst-Planck equations for the concentrations may be written as 
\begin{EqSystem}
    \partial_t c^+ + \D(c^+u) = \D(A^+ c^+ \nabla (\ln c^+ + e \psi)) \\
    \partial_t c^- + \D(c^-u) = \D(A^- c^- \nabla (\ln c^- - e \psi))
\end{EqSystem}
and so, up to the constant $A^+$ and $A^-$, we recognize the term we want to express. We get 
\begin{equation*}
    \left[  \frac{1}{A^+} \left(\partial_t c^+ + \D(c^+ u) \right) + \frac{1}{A^-} \left(\partial_t c^- + \D(c^- u) \right) \right]= \left[ \D(c^+ \nabla (\ln c^+ + e \psi) + \D(c^- \nabla (\ln c^- - e \psi) \right]
\end{equation*}
So, multiplying by $ \mu \ln\rho$ and integrating over ${\Pi^d}$, we obtain after integrating by parts
\begin{align}{\label{eq c+/- divided A^+/-}}
    \int_{\Pi^d} & \left[  \frac{1}{A^+} \left(\partial_t c^+ + \D(c^+ u) \right) + \frac{1}{A^-} \left(\partial_t c^- + \D(c^- u) \right) \right] \mu  \ln\rho \nonumber \\
    &= - \int_{\Pi^d} \left[ (c^+ \nabla (\ln c^+ + e \psi) + (c^- \nabla (\ln c^- - e \psi) \right] \mu \nabla \ln\rho
\end{align}
Then, dividing the mass equation by $\rho$, we have 
$$\partial_t \ln \rho + u \cdot \nabla \ln\rho + \D u = 0$$
Multiplying this, respectively by $\frac{\mu c^+}{A^+}$ and $\frac{\mu c^-}{A^-}$ and taking the sum, we get 
$$ \frac{1}{A^+} \left( \mu c^+ \partial_t \ln\rho + \mu c^+ u \cdot \nabla \ln \rho + \mu c^+ \D(u) \right) + \frac{1}{A^-} \left( \mu c^- \partial_t \ln\rho + \mu c^-u \cdot \nabla \ln\rho + \mu c^- \D(u) \right) = 0$$
Integrating over ${\Pi^d}$ and adding this equality to (\ref{eq c+/- divided A^+/-}), we get 
\begin{align}
    \int_{\Pi^d} & \left[ \frac{1}{A^+} \left( \partial_t (c^+ \mu \ln\rho) + \D(c^+u \mu \ln\rho + \mu c^+ \D(u) \right) + \frac{1}{A^-} \left( \partial_t (c^- \mu \ln\rho) + \D(c^-u \mu \ln\rho + \mu c^- \D(u) \right) \right] \nonumber \\
    &= - \mu \int_{\Pi^d} \left[ (c^+ \nabla (\ln c^+ + e \psi) + (c^- \nabla (\ln c^- - e \psi) \right]\cdot \nabla \ln\rho
\end{align}
Then, (\ref{BD entropy f(rho)}) can be written 
\begin{align}{\label{BD entropy I}}
    \frac{1}{2} \frac{d}{dt}  \int_{\Pi^d} \rho |u + 2\mu \nabla \log\rho|^2 & + \frac{d}{dt} \int_{\Pi^d} \rho e(\rho) + 2 \mu \int_{\Pi^d} \rho |A(u)|^2 + 2 \mu \int_{\Pi^d} \frac{p'(\rho)}{\rho} |\nabla \rho|^2 + \frac{d}{dt} \int_{\Pi^d} \varepsilon \frac{| \nabla \psi|^2}{2} \nonumber \\
    &+ \frac{d}{dt} \int_{\Pi^d} \left( \sigma (c^+) + \sigma (c^-) \right)  + \int_{\Pi^d} A^+ c^+ |\nabla (\ln c^+ + e \psi)|^2 + \int_{\Pi^d} A^- c^- |\nabla (\ln c^- - e \psi|^2 \nonumber \\
    &= 2 \left( \frac{\mu }{A^+} \frac{d}{dt} \int_{\Pi^d} c^+ \ln\rho + \frac{\mu }{A^-} \frac{d}{dt} \int_{\Pi^d} c^- \ln\rho \right) + 2 \int_{\Pi^d} \left( \frac{\mu c^+ \D(u)}{A^+} + \frac{\mu c^- \D(u)}{A^-} \right)
\end{align}
We are looking for another expression to sum with this one, to simplify the terms with $c^+ \D(u)$ and $c^-\D(u)$. \\
We write the equations for the concentrations 
\begin{align*}
    \frac{1}{A^+} \left( \partial_t c^+ + \D(c^+ u) \right) &= \D(2 \sqrt{c^+} \nabla \sqrt{c^+} + e c^+ \nabla \psi) \\
    \frac{1}{A^-} \left( \partial_t c^- + \D(c^- u) \right) &= \D(2 \sqrt{c^-} \nabla \sqrt{c^-} - e c^- \nabla \psi)
\end{align*}
We multiply the first one by $\mu \ln c^+$ and we integrate, we have 
\begin{align*}
    2 \mu \int_{\Pi^d} \D(\sqrt{c^+} \nabla \sqrt{c^+}) \ln c^+ &= - 2 \mu \int_{\Pi^d} 2 \sqrt{c^+} \nabla \sqrt{c^+}\cdot  \frac{\nabla c^+}{c^+} \\
    &= - \int_{\Pi^d} 4 \mu |\nabla \sqrt{c^+}|^2
\end{align*}

\begin{align*}
    \int_{\Pi^d} \mu \D(e c^+ \nabla \psi) \ln c^+ &= - \int_{\Pi^d} \mu c^+ e \nabla \psi \cdot \frac{\nabla c^+}{c^+} \\
    &= - \int_{\Pi^d} \mu e \nabla c^+ \cdot \nabla \psi
\end{align*}

\begin{align*}
    \frac{1}{A^+} \left( \partial_t c^+ + \D(c^+ u) \right)\mu \ln c^+ &= \frac{\mu}{A^+} \left( \partial_t c^+ \ln c^+ + (u \cdot  \nabla c^+) \ln c^+ + c^+ \D u \ln c^+ \right) \\
    &= \frac{\mu}{A^+} \left( \partial_t (c^+ \ln c^+) - c^+ \frac{\partial_t c^+}{c^+} + u \cdot \nabla (c^+ \ln c^+) - c^+ u \cdot \frac{\nabla c^+}{c^+} + c^+ \D u \text{ln}(c^+) \right) \\
    &= \frac{\mu}{A^+} \left( \partial_t (c^+ \ln c^+) + \D(u c^+ \ln c^+) - \left( \partial_t c^+ + u \cdot \nabla c^+ \right) \right)
\end{align*}
Combining these three computations, we get for $c^+$, 
$$\frac{d}{dt} \int_{\Pi^d} \frac{\mu}{A^+} \left( c^+ \ln c^+ - c^+ + 1 \right) + \int_{\Pi^d} \frac{\mu \D u c^+}{A^+} + 4 \mu \int_{\Pi^d} |\nabla \sqrt{c^+}|^2 = - \int_{\Pi^d} \mu e \nabla c^+ \cdot \nabla \psi$$
Doing the same for $c^-$, we have 
$$\frac{d}{dt} \int_{\Pi^d} \frac{\mu}{A^-} \left( c^- \ln c^- - c^- + 1 \right) + \int_{\Pi^d} \frac{\mu \D u c^-}{A^-} + 4 \mu \int_{\Pi^d} |\nabla \sqrt{c^-}|^2 = \int_{\Pi^d} \mu e \nabla c^- \cdot \nabla \psi$$
Multiplying the Poisson equation by $\mu \Delta \psi$, we get 
$$\int_{\Pi^d} \mu \varepsilon |\Delta \psi|^2 = - \int_{\Pi^d} \mu e(c^+ - c^-) \Delta \psi = \int_{\Pi^d} \mu e \nabla (c^+ - c^-) \cdot \nabla \psi$$
So, summing the equations for $c^+$ and $c^-$, we obtain 
\begin{align}{\label{eq racine c+/-}}
    \frac{d}{dt} \int_{\Pi^d} & \left( \frac{\mu}{A^+}(c^+ \ln c^+ - c^+ +1) + \frac{\mu}{A^-}(c^- \ln c^- - c^- +1) \right) + \int_{\Pi^d} \left( \frac{\mu c^+ \D(u)}{A^+} + \frac{\mu c^- \D(u)}{A^-} \right) \nonumber \\
    &+ 4\mu \int_{\Pi^d} |\nabla \sqrt{c^+}|^2 + 4\mu \int_{\Pi^d} |\nabla \sqrt{c^-}|^2 + \int_{\Pi^d} \mu \varepsilon |\Delta \psi|^2 = 0
\end{align}
We multiply (\ref{eq racine c+/-}) by 2 and add it to (\ref{BD entropy I}), we get 
\begin{align}{\label{BD entropy II}}
    \frac{1}{2} \frac{d}{dt}  \int_{\Pi^d} \rho |u + 2\mu \nabla \log\rho|^2 & + \frac{d}{dt} \int_{\Pi^d} \rho e(\rho) + \frac{d}{dt} \int_{\Pi^d} \varepsilon \frac{| \nabla \psi|^2}{2} + \frac{d}{dt} \int_{\Pi^d} \left( \sigma (c^+) + \sigma (c^-) \right) \nonumber \\
    & + \frac{d}{dt} \int_{\Pi^d} \left( \frac{2 \mu }{A^+}(c^+ \ln c^+ - c^+ +1) + \frac{2 \mu}{A^-}(c^- \ln c^- - c^- +1) \right) \nonumber \\
    &+ 8 \mu \int_{\Pi^d} |\nabla \sqrt{c^+}|^2 + 8 \mu \int_{\Pi^d} |\nabla \sqrt{c^-}|^2   + 2 \mu \int_{\Pi^d} \rho |A(u)|^2 + 2 \mu \int_{\Pi^d} \frac{p'(\rho)}{\rho} |\nabla \rho|^2  \nonumber \\
    &+ 2 \mu \int_{\Pi^d} \varepsilon |\Delta \psi|^2 + \int_{\Pi^d} A^+ c^+ |\nabla (\ln c^+ + e \psi)|^2 + \int_{\Pi^d} A^- c^- |\nabla (\ln c^- - e \psi|^2 \nonumber \\
    &= \frac{2 \mu}{A^+} \frac{d}{dt} \int_{\Pi^d} c^+ \ln \rho + \frac{2 \mu}{A^-} \frac{d}{dt} \int_{\Pi^d} c^- \ln \rho
\end{align}

As all the terms on the l.h.s are all positive, we only have to control the terms on the r.h.s. \\
Noticing that 
\begin{align*}
    c^+ \ln c^+ - c^+ +1 - c^+\ln \rho &= c^+ ( \ln c^+  - \ln\rho) - c^+ + 1 \\
    &= \rho \left( \frac{c^+}{\rho} \ln (\frac{c^+}{\rho}) - \frac{c^+}{\rho} + 1 \right) + 1 - \rho
\end{align*}
Finally, we get the conclusion plugging everything together.
\end{proof} 

\begin{Remark}
We note that we have considered here the case $\mu(\rho)= \mu\, \rho$ and $\lambda(\rho)= 0$. The more general case with $\lambda(\rho) = 2 (\mu'(\rho)\rho -\mu(\rho)).$
\end{Remark}

\section{The approximate system.}
 We have now in hand the energy and an extension of the BD entropy as possible tools. This is now time to construct approximate solutions which will satisfy the bounds given by these two quantities uniformly with respect to the parameters we will introduce to  allow  to pass to the limit with respect to all the parameters. Namely we want to build solution of the following regularized system
 \begin{EqSystem}
    \partial_t \rho + \D(\rho u) = \xi\Delta\rho , \label{eq 1 PNPNS}  \\
   \partial_t(\rho u) + {\rm div}(\rho u \otimes u) - 2\mu {\rm div}(\rho D(u)) + \nabla p(\rho) \nonumber \\
   \hskip2cm + \eta \Delta^2 u - \delta \rho \nabla \Delta^{2s+1} \rho  + \xi \nabla \rho \cdot \nabla u 
   = \varepsilon \Delta \psi \nabla \psi - \nabla(c^++c^-)
    \label{eq 2bis} \\
    \partial_t c^+ + \D(c^+ u) = \D(A^+ \nabla c^+) + \D(A^+ e c^+ \nabla \psi), \label{eq 3 PNPNS} \\
    \partial_t c^- + \D(c^- u) = \D(A^- \nabla c^-) - \D (A^- e c^- \nabla \psi), \label{eq 4 PNPNS} \\
     - \varepsilon \Delta \psi = \Psi    \hbox{ with }   (1-\zeta \Delta) \Psi = e(c^+ - c^-).   \label{eq psi pnpns} 
\end{EqSystem}
Remark that the main ingredients for regularization in the PDE above at the level of the density dependent viscosities compressible Navier-Stokes equations has been firstly proposed in \cite{BrDe1} and developed in \cite{MuPoZa} with a part of cold pressure.  For more general density dependent viscosities with also some singular pressure close to vacuum, this has been extended in \cite{BrDeZa} using the augmented formulation and a $\kappa$ mathematical entropy that will not be discussed here in our paper since we only consider the case $\mu(\rho)= \mu \rho$, $\lambda(\rho)=0$. 

\begin{Remark}
 Note that for the density dependent viscosities compressible Navier-Stokes equations, the introduction of a singular pressure close to vacuum has been removed in {\rm \cite{VaYu}} for $\mu(\rho) = \mu \, \rho$ with $\mu>0$ constant and in {\rm \cite{BrVaYu}} for more general density dependent viscosities such that $\lambda(\rho) = 2(\mu'(\rho)\rho - \mu(\rho))$ and $\lambda(\rho) + 2\mu(\rho)/d \ge \varepsilon \mu(\rho)$ for some $\varepsilon>0$. Note that in the present Poisson-Nernst-Planck system coupled by the compressible Navier-Stokes equations with degenerate viscosity, we are concerned how to define the velocity field which involves in the equations of the two concentrations $c^\pm$. A singular pressure seems thus needed as already noticed in works on chemical reacting gaseous mixture, see {\rm  \cite{MuPoZa}} and {\rm \cite{Za}}. 
 \end{Remark} 
 
 The first thing to prove is a result concerning the regularized system. Namely to prove the following theorem.
 \begin{theorem} Let $(\xi, \eta,\delta,\zeta)$ be fixed positive parameters. Assume the initial data  satisfy 
 \begin{align}\int_{\Pi^d} \Bigl[ \frac{1}{2}\frac{|m_0|^2}{\rho_0}
 & + \rho_0 e(\rho_0)\Bigr] 
  + \int_{\Pi^d} \varepsilon \frac{|\nabla\Psi_0|^2}{2} 
 +  \int_{\Pi^d} \mu^2|\nabla \sqrt\rho_0|^2 
 + \int_{\Pi^d} \varepsilon\zeta \frac{|\Delta \psi_0|^2}{2}
 \\
 & \nonumber + \int_{\Pi^d} \rho_0 \Bigl[\sigma(c^+_0/\rho_0) + \sigma(c^-_0/\rho_0)\Bigr]
 + \int_{\Pi^d} \Bigl[\sigma(c^+_0) + \sigma(c^-_0)\Bigr] 
 + \frac{\delta}{2}\int_{\Pi^d}|\nabla\Delta^s\rho_0|^2 <+\infty.
 \end{align}
 Then there exist $(\rho,u,c^+,c^-,\psi)$ solutions of System \eqref{eq 1 PNPNS}--\eqref{eq psi pnpns} where the first
 equation holds a.e. on $(0,T)\times {\Pi^d}$ together with the initial condition $\rho\vert_{t=0} = \rho^0$ and the remaining ones are satisfied in the sense of distributions on $(0,T)\times {\Pi^d}$ with the initial conditions $\rho u\vert_{t=0} = m_0$, 
 $c^\pm \vert_{t=0} = c^\pm_0$ satisfies in the sense of distributions
 on ${\Pi^d}$. Moreover, we have
  \begin{EqSystem}\label{estim} 
   \rho \in L^2(0,T;W^{2s+2,2}(\Pi^d))\cap L^\infty (0,T;W^{2s+1,2}(\Pi^d))
   \qquad \partial_t \nabla \rho \in L^2((0,T)\times \Pi^d), 
     \\
     \rho^{-1} \in L^\infty((0,T)\times \Pi^d) \hbox{ with the bound depends on }  \delta \\
    u \in L^2(0,T; W^{2,2}(\Pi^d)) \cap L^\infty(0,T; L^2(\Pi^d)) \\
    \sqrt c^\pm \in L^\infty (0,T;L^2(\Pi^d)) \cap L^2(0,T;H^1(\Pi^d))\\
    \psi \in L^2(0,T;H^3(\Pi^d)) \cap L^\infty(0,T;H^2(\Pi^d)).
\end{EqSystem}
 \end{theorem} 
 Note that as it was in \cite{MuPoZa} it is important to remark that this theorem is the first step in the
 proof of our main result namely the existence of entropy-weak solutions because with such regularity, 
 we can prove that the following estimates is true
 \begin{align}{\label{BDentropiereg}}
    \frac{1}{2} & \int_{\Pi^d} \rho |u + 2\mu \nabla \ln\rho|^2  + \int_{\Pi^d} \rho e(\rho) + \int_{\Pi^d} \varepsilon \frac{| \nabla \psi|^2}{2} 
    + \int_{\Pi^d} \varepsilon\zeta \frac{|\Delta \psi|^2}{2}
    + \int_{\Pi^d} \left( \sigma (c^+) + \sigma (c^-) \right) \nonumber \\
    & +  \int_{\Pi^d} \left[ \frac{2 \mu}{A^+}\rho \left( \frac{c^+}{\rho} \ln(\frac{c^+}{\rho}) - \frac{c^+}{\rho} +1 \right) + \frac{2 \mu}{A^-} \rho \left( \frac{c^-}{\rho} \ln (\frac{c^-}{\rho}) - \frac{c^-}{\rho} +1 \right) \right] 
    + \frac{\delta}{2} \int_{\Pi^d} |\nabla \Delta^s \rho|^2   \nonumber \\
        & + \mu \int_0^T\int_{\Pi^d} \rho |A(u)|^2   + \mu \int_0^T \int_{\Pi^d} \frac{p'(\rho)}{\rho} |\nabla\rho|^2 
        \nonumber \\
        & + 2\delta  \int_0^T \int_{\Pi^d} |\Delta^{s+1} \rho|^2  + \mu \varepsilon\int_0^T \int_{{\Pi^d}} |\Delta \psi|^2
+ \delta\xi \int_0^T\int_{\Pi^d} |\Delta^{s+1} \rho|^2 + \eta  \int_0^T\int_{\Pi^d} |\Delta u|^2 \le C.
\end{align}
Let us outline the strategy of the proof for the existence result of such regularized system that we will not do because it follows the works in  \cite{MuPoZa} and also in some sense in \cite{MaWa} for the density equation and the equation on $c^\pm$ and $\psi$ :
\begin{itemize}
\item  For a given velocity field, solve the density equation with diffusion then solve the Poisson-Nernst-Planck equation and through a Faedo-Galerkin approximation  for the weak formulation of  the momentum equation find the new velocity.
\item For suﬃciently small time interval, we find as usually the unique solution to the momentum equation applying the Banach fixed point theorem. Then we extend the existence result for the maximal time interval.  
\item We recover the full system using a version of the Leray--Schauder fixed point theorem.
\end{itemize}

\bigskip

\noindent {\bf Important remarks.}
We observe that the energy estimates and the BD entropy provides the following bounds on $\rho$
$$\rho^{-2k} \in L^\infty(0,T;L^2({\Pi^d}))\cap
    L^2(0,T;H^1({\Pi^d}))$$
uniformly with respect to the parameters.
Assuming $\delta$ to be fixed, as introduced in \cite{BrDe2} and developed also in \cite{MuPoZa}, taking $s$ suﬃciently large in \eqref{eq 2bis} we can show that the density is separated from 0 uniformly with respect to all approximation parameters except for $\delta$. Indeed, since
by the Sobolev embedding $\|\rho^{-1}\|_{L^\infty} \le C \|\rho^{-1}\|_{W^{3,k}(\Pi^d)}$  for some $k>1$
and
$$\|\nabla^3 \rho^{-1}\|_{L^k(\Pi^d)} \le (1+ \|\nabla^3\rho\|_{L^{2k}(\Pi^d)})^3(1+\|\rho^{-1}\|_{L^{4k}(\Pi^d)})^4 $$
is bounded on account of the energy estimate and the mathematical entropy;  provided that $2s+ 1 \ge 4$, we have
$$\|\rho^{-1}\|_{L^\infty((0,T)\times \Pi^d)} \le c(\delta) \hbox{ a.e. in } (0,T)\times \Pi^d.$$
In some sense, we work with no vacuum for the approximate system. Recall that from the energy and BD entropy, the bound on $\rho^{-2k}$
is uniform with respect to the parameters. This information
allows to get a bound on $u$ and $\nabla u$ from the bound $\sqrt\rho \nabla u \in L^2((0,T)\times {\Pi^d})$ and
$\sqrt\rho u \in L^2((0,T)\times {\Pi^d}).$
More precisely, we can write
$$\|\nabla u\|_{L^{p}(0,T;L^{q}({\Pi^d}))} 
\le \Bigl(1+ \|\nabla \rho^{-2k}\|_{L^2((0,T)\times{\Pi^d})} \Bigr)
\|\rho^{1/2} \nabla u\|_{L^2((0,T)\times{\Pi^d})}
$$
where 
$$1/p = 1/2 + 1/8k, \qquad
  1/q = 1/2 + 1/24k.
$$
Moreover using the Sobolev embedding, we get
$$u\in L^p(0,T;L^{q^\star}({\Pi^d}))$$
with $p=8k/(4k+1)$ and  
$q^\star = 24k/(4k+1).$
Note that when $k$ tends to $\infty$ then
$p=2, q=2$ and $q^\star=6.$ 
This information on $u$ and $\nabla u$ will be used to prove the nonlinear weak stability. We will not perform the asymptotic limits (First $\zeta$ tend to zero then $\eta, \xi$ to zero and finally $\delta$ tend to zero) from the approximate system to the Poisson-Nernst-Planck compressible Navier-Stokes system since this follows a combination of what is done for the compressible Navier-Stokes equations with degenerate viscosities for instance in \cite{BrDe1, MuPoZa, Za} and an adaptation of what is done in \cite{MaWa} for the Poisson-Nernst-Planck system  with appropriate bounds on the velocity field using the singular pressure. We focus on the weak nonlinear stability which is the most difficult part to prove.

\section{Stability / proof of the main Theorem} 
 In this section, we prove the weak nonlinear stability for the Poisson-Nernst-Planck compressible Navier-Stokes system with the shear viscosity $\mu(\rho) = \mu \, \rho$ and $\lambda(\rho)$. 
 More precisely we assume to have a sequence $(\rho_n,u_n,c^\pm_n,\psi_n)$ of global entropy weak solutions of the system satisfying uniformly the energy estimates and the  BD entropy given in Section 2. We want to prove that there exists a subsequence converging to $(\rho,u,c^\pm,\psi)$ in a sense to be precised where this limit is a global  entropy weak solution of the Poisson-Nernst-Planck compressible Navier Stokes system. This is important to note that the limit passage in the mass and momentum equations except the term $c_n^\pm \nabla \psi_n$ is the same than from the compressible Navier-Stokes equations with density dependent viscosities. In the equation related to $c^\pm$, the non-linear quantities we have to pass to the limit are
 $u_n c_n^\pm$ and $c_n^\pm \nabla\psi_n$.
 Note that, as in \cite{MaWa}, we have
$$\nabla \sqrt c_n^\pm \in L^2((0,T)\times {\Pi^d}),
\qquad 
  \sqrt c_n^\pm \in L^\infty(0,T; L^2({\Pi^d})).$$
  Therefore
$$c_n^\pm \in L^1(0,T;L^3({\Pi^d}))$$
By interpolation $\|\cdot\|_{L^r({\Pi^d})} \le \|\cdot\|^\theta_{L^3({\Pi^d})} 
 \|\cdot \|_{L^1({\Pi^d})}^{1-\theta}$
 where $0\le \theta\le 1$, we get
$$c_n^\pm \in L^{1/\theta}(0,T;L^r({\Pi^d}))
\hbox{ where } r = 3/(3-2\theta) \hbox{ and } 0<\theta\le 1.
$$
Recall (sse the section before) that we have
$$u_n\in L^p(0,T;L^{q^\star}({\Pi^d}))
\hbox{ uniformly with respect to } n$$
where $p=8k/(4k+1)$, $q^\star = 24k/(4k+1).$

\begin{Remark}
Remark that due to the degeneracy of the viscous quantity, 
we cannot prove that $u_n\in L^2((0,T)\times {\Pi^d})$
even with the singular pressure.
\end{Remark}

 Let us adapt the lines proved in \cite{MaWa} on 
$\partial_t \sqrt{c_n^\pm +1}$ showing that we have enough integrability on $u_n$ and $\nabla u_n$ to prove compactness on $c_n^\pm$. First let us prove compactness on  $\sqrt{c_n^\pm}$ observing that we already know that
$\sqrt{c_n^\pm} \in L^2((0,T)\times H^1({\Pi^d}))$.
As proved in \cite{MaWa}, we have
\begin{align}
2\partial_t \sqrt{1+c_n^\pm}
= 
& - {\rm div} 
\Bigl(\bigl(\frac{c_n^\pm}{1+c_n^\pm}  \bigr)^{1/2}\sqrt{c_n^\pm} u_n  \Bigr)
- \frac{(c_n^\pm)^{3/2}}
{(1+c_n^\pm)^{3/2}}u_n \cdot \nabla (c_n^\pm)^{1/2} 
\nonumber \\
& + {\rm div}
\Bigl(\bigl( 
\frac{c_n^\pm}{1+c_n^\pm}\bigr)^{1/2} \bigl(2 \nabla \sqrt{c_n^\pm}  \pm \sqrt{c_n^\pm} \nabla \psi_n\bigr)\Bigr)
\nonumber \\
& + \frac{c_n^\pm}{(1+c_n^\pm)^{3/2}} |\nabla\sqrt{c_n^\pm}|^2
\pm  
\Bigl(
\frac{c_n^\pm}{1+c_n^\pm}\Bigr)^{3/2} \nabla\sqrt{c_n^\pm}\cdot
\nabla \psi_n  \nonumber \\
& = J_1+J_2+J_3+J_4+J_5
\end{align}
where compared to \cite{MaWa}, we need to write $J_2$  as
\begin{align}
& \frac{(c_n^\pm)^{3/2}}
{(1+c_n^\pm)^{3/2}}u_n \cdot \nabla  (c_n^\pm)^{1/2} 
 = u_n \cdot \nabla (1+c_n^\pm)^{1/2}
+ u_n \cdot \nabla (1+c_n^\pm)^{-1/2}
\nonumber \\
& =
{\rm div}\bigl(u_n((1+c_n^\pm)^{1/2}+
(1+c_n^\pm)^{-1/2})\bigr)
- ((1+c_n^\pm)^{1/2}+
(1+c_n^\pm)^{-1/2}){\rm div} u_n.
\end{align}
Let us observe that $J_3\in L^2(0,T;H^{-1}({\Pi^d}))$ since $\sqrt{c_n^\pm}$
belongs to  $L^2(0,T; H^1({\Pi^d}))$ and $\nabla \psi$ belongs to $L^\infty(0,T;L^6({\Pi^d})).$
We also have $J_4 \in L^1((0,T)\times {\Pi^d})$ because $\sqrt{c_n^\pm} \in L^2((0,T);H^1({\Pi^d}))$.
Moreover $\nabla \psi \in L^2(0,T;H^{1}({\Pi^d}))$ and $\nabla\sqrt{c_n^\pm} \in L^2((0,T)\times {\Pi^d})$ thus $J_5\in L^1(0,T;L^{3/2}({\Pi^d}))$. Note also that ${c_n^\pm}^{1/2} u_n \in L^2((0,T)\times {\Pi^d})$ (idem $(c_n^\pm)^{1/2} {\rm div} u_n \in L^1((0,T)\times \Omega$)  because $c_n^\pm \in L^\infty(0,T;L^6({\Pi^d}))$  and 
$$\|u_n\|_{L^{p}(0,T;W^{1,p}({\Pi^d}))} \le C \hbox{ uniformly with respect to } n
$$
where 
$$1/p = 1/2 + 1/8k, \qquad
  1/q = 1/2 + 1/24k.
$$
Therefore we conclude that  $$\partial_t \sqrt{c_n^\pm +1} \in
L^1(0,T;H^{-1}({\Pi^d}))\cap L^1((0,T)\times{\Pi^d})$$
Thus using the Aubin-Lions-Simon Lemma, this provides $\sqrt{c_n^\pm+1}$ compact in $L^1(0,T;L^p({\Pi^d}))$ for all
$1<p<6.$ Then there exists a function $c^\pm$ such that 
$c_n^\pm$ strongly converges to $c^\pm$ in $L^1(0,T;L^p({\Pi^d}))$ for all $1<p<3$ recalling that we have more integrability on $c^\pm$. Therefore we conclude as in \cite{MaWa} that
$$\nabla c_n^\pm = 2 \sqrt{c_n^\pm}
\nabla \sqrt{c_n^\pm}$$ converges
weakly to $2 \sqrt{c^\pm}\nabla
\sqrt{c^\pm} = \nabla c^\pm$
in $L^2(0,T;L^1(\Pi^d))\cap L^1(0,T;L^q(\Pi^d))$ for all $q<3/2$.
Moreover we know that $\nabla\psi_n$ is uniformly bounded in $L^\infty(0,T;H^1({\Pi^d})).$
  Using the strong convergence on $c_n^\pm$ coupled with the weak convergence in $\nabla\psi_n$, we can pass to the limit in $c_n^\pm \nabla \psi_n$. Finally using the uniform bound on $u_n$ we have given in the previous section and the strong convergence on $c^\pm_n$, we can pass to the limit in the term $c_n^\pm u_n$. This ends the proof related to the weak nonlinear stability.


\bigskip

\noindent {\bf Acknowledgments.}   The first author gratefully acknowledges the partial support by the Agence
Nationale pour la Recherche grant ANR-23-CE40-0014-01 (ANR Bourgeons). This work also benefited of the support of the ANR under France 2030 bearing the reference ANR-23-EXMA-0004 (Complexflows project).  The third author gratefully acknowledges the PEPR MathsViVEs in France managed by A. Guillin and V. Calvez for a  three years doctoral fellowship. This work received support from the French government, managed by the National Research Agency (ANR), under the France 2030 program, reference ANR-23-EXMA-0001.The authors want to thank E. Charlaix, M. Meyer and C. Picard (Liphy Grenoble) for the project (physical modeling, numerical development and experiments) related to the saline gradient energy we are working on together.


\end{document}